\documentclass[a4paper, 12pt, titlepage, fleqn]{article}
\usepackage{times}
\usepackage{amsthm}
\newtheorem{thm}{Theorem}
\newtheorem{cor}{Corollary}

\newtheorem{lemma}[thm]{Lemma}
\newtheorem{prop}{Proposition}

\usepackage{amssymb,amsmath}

\DeclareMathOperator{\F}{\mathbb{F}}

\begin{document}
		\baselineskip=16.3pt
		\parskip=14pt
		\begin{center}
			\section*{On the Zeta Functions of Supersingular Curves}
			{\large 
			Gary McGuire\footnote{email gary.mcguire@ucd.ie, Research supported by Science Foundation Ireland Grant 13/IA/1914} and
			Emrah Sercan Y{\i}lmaz \footnote {Research supported by Science Foundation Ireland Grant 13/IA/1914} 
			\\
			School of Mathematics and Statistics\\
			University College Dublin\\
			Ireland}

		\end{center}

	\subsection*{Abstract}
	In general, the L-polynomial of a curve of genus $g$ is determined by $g$ coefficients.
We show that the L-polynomial of a supersingular curve of genus $g$
is determined by fewer than $g$ coefficients.

\textbf{Keywords:} Zeta Functions; L-polynomials; Supersingular Curves.

\section{Introduction}
Let $p$ be a prime and $r \ge 1$ be an integer. Throughout this paper we let $q=p^r$. 
Let $X$ be a projective smooth absolutely irreducible curve of genus $g$ defined over $\mathbb{F}_q$.
Consider the $L$-polynomial of the curve $X$ over $\mathbb F_{q}$, defined by
$$L_{X/\mathbb{F}_q}(T)=L_X(T)=\exp\left( \sum_{n=1}^\infty ( \#X(\mathbb F_{q^n}) - q^n - 1 )\frac{T^n}{n}  \right)$$
where $\#X(\mathbb F_{q^n})$ denotes the number of $\mathbb F_{q^n}$-rational points of $X$. 
The L-polynomial is the numerator of the zeta function of $X$.
It is well known that $L_X(T)$ is a polynomial of degree $2g$ with integer coefficients, so we write it as 
\begin{equation} \label{L-poly}
L_X(T)= \sum_{i=0}^{2g} c_i T^i, \ c_i \in \mathbb Z.
\end{equation}
It is also well known that $c_0=1$ and $c_{2g-i}=q^{g-i}c_i$ for $i=0,\cdots,g$. Let $\eta_1,\cdots, \eta_{2g}$ be the roots of the reciprocal of the $L$-polynomial of $X$ over $\mathbb F_q$ (sometimes called the Weil numbers of $X$, or the
Frobenius eigenvalues). For any $n\ge 1$ we have \begin{equation}\label{eqn-root-sum}
\#X(\F_{q^n})-(q^n+1)=-\sum_{i=1}^{2g} \eta_i^{n}.
\end{equation}
We refer the reader to \cite{AFF-HS} for all background on curves.

A curve $X$ of genus $g$ defined over $\mathbb F_q$ is \emph{supersingular} if any of the following equivalent properties hold.

\begin{enumerate}
	\item All Weil numbers of $X$ have the form $\eta_i = \sqrt{q}\cdot \zeta_i$ where $\zeta_i$ is a root of unity.
	\item The Newton polygon of $X$ is a straight line of slope $1/2$.
	\item The Jacobian of $X$ is geometrically isogenous to $E^g$ where
	$E$ is a supersingular elliptic curve.
	\item If $X$ has $L$-polynomial
	$L_X(T)=1+\sum\limits_{i=1}^{2g} c_iT^i$ 
	then
	$$ord_p(c_i)\geq \frac{ir}{2}, \ \mbox{for all $i=1,\ldots	,2g$.}$$
\end{enumerate}

Let $ \sqrt{q}\cdot \zeta_1, \ldots ,  \sqrt{q}\cdot \zeta_{2g}$
be the Weil numbers of a supersingular curve $X$, where the
$\zeta_i$ are roots of unity.
By equation \eqref{eqn-root-sum} for any $n\ge 1$ we have \begin{equation}\label{eqn-root-sum-2}
-q^{-n/2}[\#X(\F_{q^n})-(q^n+1)]=\sum_{i=1}^{2g} \zeta_i^{n}.
\end{equation}

The smallest positive integer $s=s_X$ such that $\zeta_i^{s}=1$ for all 
 $i=1,\ldots	,2g$ will be called the \emph{period} of $X$. 
The period depends on $q$, in the sense that $X(\F_{q^n})$
may have a different period to $X(\F_{q})$.
Note that this is slightly different from the period as defined in \cite{KP}.

In this paper we will prove the following theorems. The first theorem is our main theorem.

\newpage

\begin{thm}\label{reduction-thm}
	Let $X$ be a supersingular curve of genus $g$ defined over $\mathbb F_q$ with
	period $s$.
	Let $n$ be a positive integer, let $\gcd (n,s)=m$ and write $n=m\cdot t$. If $q$ is odd, then we have $$
	\#X(\F_{q^n})-(q^n+1)=\begin{cases}
	q^{(n-m)/2}[\#X(\F_{q^m})-(q^m+1)] &\text{if } m\cdot r \text{ is even},	\\
	q^{(n-m)/2}[\#X(\F_{q^m})-(q^m+1)]\left(\frac{(-1)^{(t-1)/2}t}{p}\right)&\text{if } m \cdot r \text{ is odd and } p\nmid t,
	\\
		q^{(n-m)/2}[\#X(\F_{q^m})-(q^m+1)]&\text{if } m \cdot r \text{ is odd and } p\mid t.
	\end{cases}$$ If $q$ is even, then we have  $$
	\#X(\F_{q^n})-(q^n+1)=\begin{cases}
	q^{(n-m)/2}[\#X(\F_{q^m})-(q^m+1)] &\text{if } m\cdot r \text{ is even},\\
	q^{(n-m)/2}[\#X(\F_{q^m})-(q^m+1)](-1)^{(t^2-1)/8}&\text{if } m \cdot r \text{ is odd}.\\
	\end{cases}$$
\end{thm}

\begin{thm}\label{thm-l-poly}
	Let $X$ be a supersingular curve of genus $g$ defined over $\mathbb F_q$
	with period $s$. Let $L_X(T)= \sum_{i=0}^{2g} c_i T^i$ be the $L$-polynomial of $X$. Assume we know $c_j$ for $1\le j <l \le g$ where $l \nmid s$. Then $c_l$ is determined. In particular, if we know  $c_j$ for $1\le j \le g$ where $j \mid s$, then all coefficients of the $L$-polynomial of $X$ are determined.
\end{thm}

In Section 2 we present the background we will need, which includes some
basic results on quadratic subfields of cyclotomic fields and Gauss sums.
In Section 3 we present the proof of Theorem 1, and in Section 4 
the proof of Theorem 2.
Section 5 contains some applications of our results.

\section{Quadratic Fields as Subfields of a Cyclotomic Field}

Let $n$ be a positive integer. 
Let $w_n$ be a primitive $n$-th root of unity which we may take to be $e^{\frac{2\pi i}{n}}$. 
We call $\Phi_n$ be the $n$-th cyclotomic polynomial and  the set of the roots of $\Phi_n$ is 
$$\{ w_n^i \: | \: 1\le i \le n \text{ and } (i,n)=1\}.$$  It is well-known that $\Phi_n$ is irreducible over $\mathbb Q$.
For odd primes $p$, define $$\sqrt{p^*}=\begin{cases}
\sqrt{p} &\text{if } p \equiv 1 \mod 4,\\
i\sqrt{p} &\text{if } p \equiv 3 \mod 4.
\end{cases}$$
The following propositions are useful for our proofs. 
\begin{prop}\label{intersection}
Let $n$ and $m$ be positive integers with $(n,m)=d$. Then we have $$\mathbb Q(w_n) \cap \mathbb Q(w_m)=\mathbb Q(w_d).$$
\end{prop}
\begin{proof}
	Let $f$ be the least common multiple of $n$ and $m$. Then we have $w_n,w_m\in \mathbb Q(w_f)$. Since $d=\frac{nm}f$, there exists $a,b \in \mathbb Z$ such that $an+bm=\frac{nm}f$ or $\frac am+\frac bn=\frac1f$. Therefore $w_f=w_n^bw_m^a \in \mathbb Q(w_n,w_m)$. Hence we have $\mathbb Q(w_n,w_m)=\mathbb Q(w_f)$.
	
	Since $d\mid n,m$, we have $w_d \in \mathbb Q(w_n) \cap \mathbb Q(w_m)$. Since $\mathbb Q(w_n)$ is a normal extension over $\mathbb Q$, we have $$[\mathbb Q(w_n)\ : \ \mathbb Q(w_n)\cap \mathbb Q(w_m)]=[\mathbb Q(w_f): \mathbb Q(w_m)]=\frac{\phi(f)}{\phi(m)}=\frac{\phi(n)}{\phi(d)}.$$ Therefore, we have $[\mathbb Q(w_n)\cap \mathbb Q(w_m): \mathbb Q]= \phi(d)$. Since  $\mathbb Q(w_d) \subseteq  \mathbb Q(w_n)\cap \mathbb Q(w_m)$ and $[ \mathbb Q(w_d): \mathbb Q]=\phi(d)$, we have $\mathbb Q(w_n) \cap \mathbb Q(w_m)=\mathbb Q(w_d)$.
\end{proof}

\begin{prop}\label{gauss-sum}
If $p$ is an odd prime, then $$\sqrt {p^*}=\sum_{j=0}^{p-1}\left(\frac{j}{p}\right)w_p^j \in \mathbb Q(w_p).$$ Moreover, if $p\equiv 3\mod 4$, then $\sqrt{p} \not \in \mathbb Q(w_p)$.
\end{prop}

\begin{proof}
	The first statement is the quadratic Gauss sum result, see \cite{Cyc-LW} for example.
	
	Let $p\equiv 3 \mod 4$ and assume $\sqrt{p} \in \mathbb Q(w_p)$. Since $i\sqrt{p}$  is in $\mathbb Q(w_p)$, then $i \in \mathbb Q(w_p)$. On the other hand,  $\mathbb Q(w_4) \cap \mathbb Q(w_p)=\mathbb Q(w_2)=\mathbb Q$ by Proposition \ref{intersection} which contradicts  $i \in \mathbb Q(w_p)$.
\end{proof}

\begin{lemma}\label{element-in-lemma}
Let $p \equiv 1\mod 4$ be a prime and $n$ be a positive integer. Then the element $\sqrt{p}$ is in $\mathbb Q(w_n)$ if and only if $p\mid n$. 
\end{lemma} 

\begin{proof}
	If $p$ does not divide $n$, we have $(n,p)=1$ and $$\mathbb Q(w_n) \cap \mathbb Q(w_p) =\mathbb Q$$ by Proposition \ref{intersection}.  Moreover, we have $$\sqrt {p}=\sum_{j=0}^{p-1}\left(\frac{j}{p}\right)w_p^j \in \mathbb Q(w_p)$$ by Proposition \ref{gauss-sum}. Therefore $$\sqrt{p} \not \in \mathbb Q(w_n).$$  On the other hand, assume $p\mid n$. Then we have $$\sqrt {p}=\sum_{j=0}^{p-1}\left(\frac{j}{p}\right)w_p^j=\sum_{j=0}^{p-1}\left(\frac{j}{p}\right)w_n^{jn/p} \in \mathbb Q(w_n)$$  by Proposition \ref{gauss-sum}. 
\end{proof}

\begin{lemma}\label{element-in-lemma-3(4)}
	Let $p\equiv 3 \mod 4$ be a prime and $n$ be a positive integer. Then the element $\sqrt{p}$ is in $\mathbb Q(w_n)$ if and only if $4p\mid n$. 
\end{lemma} 

\begin{proof}
	 Assume $4p\mid n$. Since $-i$ and $i\sqrt{p}$ are in $\mathbb Q(w_n)$, we have $\sqrt p= -i\cdot i\sqrt{p}\in \mathbb Q(w_n)$.
	   
	If $4p$ does not divide $n$, then $(n,4p)$ is $1,\ 2,\ 4,\ p$ or $2p$. If $(n,4p)$ is $1,\ 2,\ p$ or $2p$, then $$\mathbb Q(w_n) \cap \mathbb Q(w_{4p}) =\mathbb Q(w_{(n,4p)}) \subseteq \mathbb Q(w_{2p})=\mathbb Q(w_p) $$ by Proposition \ref{intersection}. Since $\sqrt{p}\in \mathbb Q(w_{4p})$ and $\sqrt p \not \in \mathbb Q(w_p)$  by Proposition \ref{gauss-sum}, we have $\sqrt p \not \in \mathbb Q(w_n)$.   If $(n,4p)=4$, then $$\mathbb Q(w_n) \cap \mathbb Q(w_{4p}) =\mathbb Q(w_4)=\mathbb Q[i].$$ So by Proposition \ref{intersection}. Since $\sqrt{p}\in \mathbb Q(w_{4p})$ and $\sqrt p \not \in \mathbb Q[i]$, we have $\sqrt p \not \in \mathbb Q(w_n)$. 
\end{proof}

\begin{lemma}
	Let $n$ be a positive integer. The element $\sqrt 2$ is in $\mathbb Q(w_n)$ if and only if $8\mid n$. 
\end{lemma}
\begin{proof}
	Assume $8$ divides $n$ and write $n=8t$ where $t$ is a positive integer. Then we have $$\sqrt {2}=w_8-w_8^3=w_n^{t}-w_n^{3t}\in \mathbb Q(w_n).$$ On the other hand, assume $8$ does not divide $n$. Then $$\mathbb Q(w_n)\cap \mathbb Q(w_8)  \subseteq \mathbb Q(w_4).$$ Since $\sqrt{2} \in \mathbb Q(w_8)$ and $\sqrt2 \not \in \mathbb Q(w_4)=\mathbb Q[i]$, we have $\sqrt2\not\in\mathbb Q(w_n)$. 
\end{proof}
\begin{lemma}
	Let $p$ be a prime and $n$ be a positive integer such that $\sqrt{p}\in \mathbb Q(w_n)$. Then the extension degree $\left[\mathbb Q(w_n) :\mathbb Q(\sqrt{p}) \right]$ is $\frac{\phi(n)}{2}$.
\end{lemma}
\begin{proof}
	Since $\sqrt{p} \in \mathbb Q(w_n)$, we have $$\mathbb Q\subset \mathbb Q (\sqrt{p}) \subset \mathbb Q(w_n).$$ Hence the result follows by the equality \begin{align*}
	\phi(n)&=\left[ \mathbb Q(w_n): \mathbb Q\right]\\[10pt]&=\left[ \mathbb Q(w_n): \mathbb Q(\sqrt{p})\right]\cdot \left[ \mathbb Q(\sqrt{p}): \mathbb Q\right]\\[10pt]&=2\left[ \mathbb Q(w_n): \mathbb Q(\sqrt{p})\right].
	\end{align*}
\end{proof}

 Let $p$ be an odd prime and $n$ be an integer such that $\sqrt p \in \mathbb Q(w_n)$. Define the index sets $I_n^+$ and $I_n^-$ as follows 
 \[
 I_n^+=\left\{ k \: | \: (k,n)=1, \; \left(\frac{(-1)^{(k-1)/2}k}{p}\right)=1 \text{ and } 1 \le k \le n  \right\}
 \]
 and 
 \[
 I_n^-=\left\{ k \: | \: (k,n)=1, \; \left(\frac{(-1)^{(k-1)/2}k}{p}\right)=-1 \text{ and } 1 \le k \le n  \right\}.
 \]
 
In same manner, let $p=2$ and  $n$ be an integer divisible by $8$. Define the index sets $I_n^+$ and $I_n^-$ as follows $$I_n^+=\{ k \: | \: (k,n)=1, \; k\equiv \pm 1 \mod 8 \text{ and } 1 \le k \le n  \}$$and$$ I_n^-=\{ k \: | \: (k,n)=1, \;k\equiv \pm 3 \mod 8  \text{ and } 1 \le k \le n  \}.$$ Moreover, define $I_n=I_n^+\cup I_n^-$. Define the polynomials $\Phi_n^{+}$ and $\Phi_n^{-}$ as follows $$\Phi_n^{+}(x)=\prod_{j\in I_n^+}(x-w_n^j) \;\;\; \text{ and } \;\;\;
 \Phi_n^{-}(x)=\prod_{j\in I_n^-}(x-w_n^j).$$

Define $G_n$ to be the Galois group $ Gal(\mathbb Q(w_n)/\mathbb Q).$
 Define  the group $G_n^+$ and the set $G_n^{-}$ as follows $$G_n^+:=\{\sigma \in G_n | \sigma(w_n)=w_n^k \text{ where } k \in I_n^+\}$$ and $$ G_n^-:=\{\sigma \in G_n | \sigma(w_n)=w_n^k \text{ where } k\in I_n^-\}.$$

 The following lemmas show that $G_n^+$ is a subgroup of the Galois group $G_n$ of index $2$ and shows that the subset $G_n^{-}$ is the relative coset of $G_n^+$ inside $G_n$.
 
 \begin{lemma}\label{auto+-p}
	Let $p \equiv 1 \mod 4$ be a prime and  $n$ be a positive integer divisible by $p$.  Then the group  $G_n^+$ fixes $\sqrt{p}$ and $G_n^{-}$ takes $\sqrt{p}$ to $-\sqrt{p}$.
 \end{lemma}
 \begin{proof} 
For a positive integer $k$ we have $\left(\frac{(-1)^{(k-1)/2}k}{p}\right)=\left(\frac{k}{p}\right)$ since $\left(\frac{-1}{p}\right)=1$.

Let $\sigma \in G_n^+$. Then there exists an integer  $k$ with $(k,n)=1$, $\left(\frac{k}{p}\right)=1$ and $1\le k \le n$ and $\sigma(w_n)=w_n^k$. Then $$\sigma(\sqrt{p})=\sigma\left(\sum_{j=0}^{p-1} \left(\frac{j}{p}\right)w_n^{jn/p}\right)=\sum_{j=0}^{p-1} \left(\frac{j}{p}\right)w_n^{kjn/p}=\sum_{j=0}^{p-1} \left(\frac{kj}{p}\right)w_n^{kjn/p}
$$
$$=\sum_{j=0}^{p-1} \left(\frac{j}{p}\right)w_n^{jn/p}=\sqrt{p}.$$
Let $\sigma \in G_n^-$. Then there exists an integer  $k$ with $(k,n)=1$, $\left(\frac{k}{n}\right)=-1$ and $1\le k \le n$ and $\sigma(w_n)=w_n^k$. Then $$\sigma(\sqrt{p})=\sigma\left(\sum_{j=0}^{p-1} \left(\frac{j}{p}\right)w_n^{jn/p}\right)=\sum_{j=0}^{p-1} \left(\frac{j}{p}\right)w_n^{kjn/p}=\sum_{j=0}^{p-1} -\left(\frac{kj}{p}\right)w_n^{kjn/p}
$$
$$=-\sum_{j=0}^{p-1} \left(\frac{j}{p}\right)w_n^{jn/p}=-\sqrt{p}.$$
\end{proof}

 \begin{lemma}\label{auto+-p-3(4)}
	Let $p \equiv 3 \mod 4$ be a prime and  $n$ be a positive integer divisible by $4p$.  Then the group  $G_n^+$ fixes $\sqrt{p}$ and $G_n^{-}$ takes $\sqrt{p}$ to $-\sqrt{p}$.
\end{lemma}

\begin{proof}

Let $\sigma \in G_n$. Then there exists an integer  $k$ with  $(k,n)=1$ and $1\le k \le n$ and $\sigma(w_n)=w_n^k$.

Since $n$ is even, $k$ is odd. Since $n$ is divisible by $p$, $k$ is not divisible by $p$. 

Moreover, we have $$\sigma(-i)=-\sigma(i)=-\sigma(w_{n}^{n/4})=-\sigma(w_n)^{n/4}=-(w_n^k)^{n/4}=-(w_n^{n/4})^k=-i^k.$$

If $\left(\frac kp\right)=1$, we have
 $$\sigma(\sqrt{p})=\sigma(-i\cdot i\sqrt{p})=\sigma(-i)\cdot \sigma\left(\sum_{j=0}^{p-1} \left(\frac{j}{p}\right)w_n^{jn/p}\right)=-i^k\sum_{j=0}^{p-1} \left(\frac{j}{p}\right)w_n^{kjn/p}$$
 $$=-i^k\sum_{j=0}^{p-1} \left(\frac{kj}{p}\right)w_n^{kjn/p}
	=-i^k\sum_{j=0}^{p-1} \left(\frac{j}{p}\right)w_n^{jn/p}=-i^k\cdot i\sqrt{p}=-i^{k+1}\sqrt{p}=(-1)^{(k-1)/2}\sqrt{p}$$ and if $\left(\frac kp\right)=-1$, we have
	$$\sigma(\sqrt{p})=\sigma(-i\cdot i\sqrt{p})=\sigma(-i)\cdot \sigma\left(\sum_{j=0}^{p-1} \left(\frac{j}{p}\right)w_n^{jn/p}\right)=-i^k\sum_{j=0}^{p-1} \left(\frac{j}{p}\right)w_n^{kjn/p}$$
	$$=-i^k\sum_{j=0}^{p-1} -\left(\frac{kj}{p}\right)w_n^{kjn/p}
	=i^k\sum_{j=0}^{p-1} \left(\frac{j}{p}\right)w_n^{jn/p}=-i^k\cdot i\sqrt{p}=i^{k+1}\sqrt{p}=-(-1)^{(k-1)/2}\sqrt{p}.$$ Now the result follows by the fact $\left(\frac{-1}{p}\right)=-1$. 
\end{proof}

 \begin{lemma}\label{auto+-2}
	Let $p=2$  and  $n$ be a positive integer divisible by $8$.  Then the group  $G_n^+$ fixes $\sqrt{2}$ and $G_n^{-}$ takes $\sqrt{2}$ to $-\sqrt{2}$.
\end{lemma}
\begin{proof}
	Let $\sigma \in G_n^+$. Then there exists an integer  $k$ with $(k,n)=1$, $k\equiv \pm 1\mod 8$ and $1\le k \le n$ and $\sigma(w_n)=w_n^k$. Then $$\sigma(\sqrt{2})=\sigma\left(w_8-w_8^3\right)=w_8^k-w_8^{3k}=\begin{cases}
	w_8-w_8^3 &\text{ if } k \equiv 1 \mod 8 \\
		w_8^7-w_8^5 &\text{ if } k \equiv 7 \mod 8 
	\end{cases}=\sqrt2.$$
	Let $\sigma \in G_n^-$. Then there exists an integer  $k$ with $(k,n)=1$, $k\equiv \pm 3\mod 8$ and $1\le k \le n$ and $\sigma(w_n)=w_n^k$. Then $$\sigma(\sqrt{2})=\sigma\left(w_8-w_8^3\right)=w_8^k-w_8^{3k}=\begin{cases}
	w_8^3-w_8 &\text{ if } k \equiv 3 \mod 8 \\
	w_8^5-w_8^7 &\text{ if } k \equiv 5 \mod 8 
	\end{cases}=-\sqrt2.$$
\end{proof}

\begin{cor}\label{irred+-}
		Let $p$ be a prime and $n$ be a positive integer such that $\sqrt p \in \mathbb Q(w_n)$. The polynomials $\Phi_n^+$ and $\Phi_n^-$ are irreducible over $\mathbb Q(\sqrt{p})$.
\end{cor}
\begin{proof}
	By Lemma \ref{auto+-p}, \ref{auto+-p-3(4)} and \ref{auto+-2}, the Galois group $Gal(\mathbb Q(w_n)/\mathbb Q(\sqrt{p}))$ is $G_n^+$. The result follows by this fact. 
\end{proof}

\begin{cor}\label{sqrt-+cor}
	Let $p$ be an odd prime and $n$ be a positive integer such that $\sqrt p \in \mathbb Q(w_n)$. Let $\ell$ be an integer. 
	There exist rational numbers $a$ and $b$ such that $$  \sum_{j\in I_n^+} w_n^{j\ell}=a+b\sqrt{p} \;\;\; \text{ and } \;\;\; \sum_{j\in I_n^-} w_n^{j\ell}=a-b\sqrt{p}.$$
\end{cor}

\begin{proof}
	Since $G_n^+$ fixes both sums, they are in $\mathbb Q(\sqrt{p})$. In other words, there exist rational numbers $a$, $b$, $c$ and $d$ such that $$\sum_{j\in I_n^+} w_n^{j\ell}=a+b\sqrt{p} \;\;\; \text{ and } \;\;\; \sum_{j\in I_n^-} w_n^{j\ell}=c+d\sqrt{p}.$$ Since $G_n^-$ sends one to the other, they are conjugate in $\mathbb Q(\sqrt{p})$. Hence $c=a$ and $d=-b$.
\end{proof}

\section{Proof of  Theorem \ref{reduction-thm}}
Let $X$ be a supersingular curve of genus $g$ defined over $\mathbb F_q$ having period $s$. Let $n$ be a positive integer and let $m=\gcd (s,n)$.

Since $X$ is also a curve of genus $g$ over $\mathbb F_{q^m}$, we will consider 
$X$ on $\mathbb F_{q^m}$. 
Write $s=m\cdot u$. Then  $\sqrt{q^m}$ times the roots of $L_{X/\F_{q^m}}$ are $u$-th roots of unity, i.e., the period of $X$ is $u$ over $\mathbb F_{q^m}$ because of equation (\ref{eqn-root-sum-2}).

Define $M_X(T)$ to be $L_{X/\F_{q^m}}(q^{-m/2}T)$.   
Then $M_X$ is monic and the roots of $M_X$ are $\zeta_1^{-m},\cdots,\zeta_{2g}^{-m}$ where $\zeta_i$'s are in equation \eqref{eqn-root-sum-2}. Hence the smallest positive integer $k$ such that $w^k=1$ for all roots $w$ of $M$ is $u$.

Write $n=m\cdot t$. Then we have $(u,t)=1$ and the extension degree $[\mathbb F_{q^n}: \mathbb F_{q^m}]=t$.
In the proofs below, we will reduce the case $t>1$ to the case $t=1$.
We will be using the fact that the roots of the polynomial $L_{X/\F_{q^n}}(q^{-n/2}T)$
of $X/\F_{q^n}$ are the $t$-th powers of the roots of $M_X(T)$.

\subsection{Proof of Theorem \ref{reduction-thm} for $r$ or $m$  even}

 Assume that either $r$ or $m$ is even. Then $M_X(T) \in \mathbb Q[T]  $. 
 Since the roots of $M_X$ are $u$-th roots of unity and $M_X$ is monic, the factorization of $M_X$ in $\mathbb Q[T]$ is as follows:  \begin{equation}\label{M-even}
 M_X(T)=\prod_{d\mid u}\Phi_d(T)^{e_{d}}
 \end{equation} where $e_d$ is a non-negative integer for each $d\mid u$.

  Since $(u,t)=1$, the map $x\to x^t$ permutes the roots of $\Phi_d$ where $d \mid u$. Therefore, we have (by equations \eqref{eqn-root-sum-2} and \eqref{M-even}) \begin{align*}
	-q^{-n/2}[\#X(\F_{q^n})-(q^n+1)]&=\sum_{d\mid u} \left( e_{d}\sum_{j\in I_d} w_d^{-jt}  \right)\\
	&=\sum_{d\mid u} \left( e_{d}\sum_{j\in I_d} w_d^{-j}  \right)\\
	&=-q^{-m/2}[\#X(\F_{q^m})-(q^m+1)].
	\end{align*}

\subsection{Proof of Theorem \ref{reduction-thm} for $r$  and $m$ odd} 
Assume $r$ and $m$ are odd. Note that $u$ must be even, because equality holds in the Hasse-Weil bound.

We have $M_X(T) \in \mathbb Q(\sqrt{p})[T]  
$. Since the roots of $M_X$ are $u$-th roots of unity, we can write $M_X(T)$ as\begin{equation}\label{M-odd}
 M_X(T)=\prod_{d \mid u, \sqrt p \in  \mathbb Q(w_d)}\Phi_d^{+}(T)^{e_{d,1}}\Phi_d^{-}(T)^{e_{d,2}} \prod_{d \mid u, \sqrt p \not \in  \mathbb Q(w_d)}\Phi_d(T)^{e_{d}}
\end{equation} where $e_{d,1}$, $e_{d,2}$ and $e_d$ are non-negative integers for each $d\mid u$.

Since $(u,t)=1$, the map $x\to x^t$ permutes the roots of $\Phi_d$ where $d \mid u$.

	{\bf Case A.} When $p$ does not divide $t$. 
	
	 We have  (by Equation \ref{eqn-root-sum-2} and \ref{M-odd}) \begin{align*}
	 -q^{-n/2}[\#X(F_{q^n})-(q^n+1)]&=\sum_{d\mid u,\sqrt p  \in  \mathbb Q(w_d)} \left( e_{d,1}\sum_{j\in I_d^+} w_d^{-jt} + e_{d,2}\sum_{j\in I_d^-} w_d^{-jt}\right)\\
	 &+\sum_{d\mid u,\sqrt p \not \in  \mathbb Q(w_d)} \left( e_{d}\sum_{j\in I_d} w_d^{-jt}  \right).
	 \end{align*}
	 	 
	{\bf Case 1.} 
	Assume that $t$ is a positive integer such that $\begin{cases}
		\left(\frac{(-1)^{(t-1)/2}t}{p}\right)=1 &\text {if } p \text{ is odd},\\
		t\equiv \pm 1 \mod  8  &\text{if } p=2.
 	\end{cases}$

 	 By Lemmas \ref{auto+-p}, \ref{auto+-p-3(4)} and \ref{auto+-2} the map $x\to x^t$ permutes the roots of $\Phi_d^+$ and  permutes the roots of $\Phi_d^-$ where $d \mid u$. 
	 Hence we have 
	  \begin{align*}
	-q^{-n/2}[\#X(F_{q^n})-(q^n+1)]&=\sum_{d\mid u,\sqrt p \in  \mathbb Q(w_d)} \left( e_{d,1}\sum_{j\in I_d^+} w_d^{-jt} 
	+ e_{d,2}\sum_{j\in I_d^-} w_d^{-jt}\right)\\
	&+\sum_{d\mid u, \sqrt p \not \in  \mathbb Q(w_d)} \left( e_{d}\sum_{j\in I_d} w_d^{-jt}  \right)\\
	&=\sum_{d\mid u,\sqrt p \in  \mathbb Q(w_d)} \left( e_{d,1}\sum_{j\in I_d^+} w_d^{-j} + e_{d,2}\sum_{j\in I_d^-} w_d^{-j}\right)\\
	&+\sum_{d\mid u, \sqrt p \not \in  \mathbb Q(w_d)} \left( e_{d}\sum_{j\in I_d} w_d^{-j}  \right)\\
	&=-q^{-m/2}[\#X(F_{q^m})-(q^m+1)].
	\end{align*}

		{\bf Case 2.} Assume that $t$ is a positive integer such that $\begin{cases}
	\left(\frac{(-1)^{(t-1)/2}t}{p}\right)=-1 &\text {if } p \text{ is odd},\\
	t\equiv \pm 3 \mod  8  &\text{if } p=2.
	\end{cases}$

By Lemmas \ref{auto+-p}, \ref{auto+-p-3(4)} and \ref{auto+-2} the map $x\to x^t$ sends the roots of $\Phi_d^+$ to the roots of $\Phi_d^-$ and vice versa where $d \mid u$. Hence we have 
 \begin{align*}
-q^{-n/2}[\#X(\F_{q^n})-(q^n+1)]&=\sum_{d\mid u,\sqrt p \in  \mathbb Q(w_d)} \left( e_{d,1}\sum_{j\in I_d^+} w_d^{-jt} 
+ e_{d,2}\sum_{j\in I_d^-} w_d^{-jt}\right)\\
&+\sum_{d\mid u, \sqrt p \not \in  \mathbb Q(w_d)} \left( e_{d}\sum_{j\in I_d} w_d^{-jt}  \right)\\
&=\sum_{d\mid u,\sqrt p \in  \mathbb Q(w_d)} \left( e_{d,1}\sum_{j\in I_d^-} w_d^{-j} + e_{d,2}\sum_{j\in I_d^+} w_d^{-j}\right)\\
&+\sum_{d\mid u, \sqrt p \not \in  \mathbb Q(w_d)} \left( e_{d}\sum_{j\in I_d} w_d^{-j}  \right).
\end{align*}

	For any $d$ with $d\mid u$, we write $\sum_{j\in I_d^+} w_d^{-j}=a_d+b_d\sqrt{p}$ where $a_d$ and $b_d$ are rational numbers  (such rational numbers exist by Corollary \ref{sqrt-+cor}). Then we have  $\sum_{j\in I_d^-} w_d^{-j}=a_d-b_d\sqrt{p}$ by Corollary \ref{sqrt-+cor}. Therefore, we 
	also have $\sum_{j\in I_d} w_d^{-j}=2a_d$. Hence the last line
	  \begin{align*}
&\sum_{d\mid u,\sqrt p \in  \mathbb Q(w_d)} \left( e_{d,1}\sum_{j\in I_d^-} w_d^{-j} + e_{d,2}\sum_{j\in I_d^+} w_d^{-j}\right)
+\sum_{d\mid u, \sqrt p \not \in  \mathbb Q(w_d)} \left( e_{d}\sum_{j\in I_d} w_d^{-j}  \right)\\
	&=\sum_{d\mid u,\sqrt p \in  \mathbb Q(w_d)} \left( e_{d,1}(a_d-b_d\sqrt{p}) + e_{d,2}(a_d+b_d\sqrt{p})\right)
+\sum_{d\mid u, \sqrt p \not \in  \mathbb Q(w_d)} \left( e_{d}\cdot 2a_d \right).
\end{align*}

		We have therefore shown that 
		\begin{align*}
			-q^{-n/2}[\#X(\F_{q^n})-(q^n+1)]
			&=\sum_{d\mid u,\sqrt p \in  \mathbb Q(w_d)} \left( e_{d,1}(a_d-b_d\sqrt{p}) + e_{d,2}(a_d+b_d\sqrt{p})\right)\\
			&+\sum_{d\mid u, \sqrt p \not \in  \mathbb Q(w_d)} \left( e_{d}\cdot 2a_d \right).
		\end{align*}
		
Note that $n$ is odd because $s$ is even and $m$ is odd.
Also $r$ is odd (where $q=p^r$) and $\#X(\F_{q^n})-(q^n+1)$ is an integer, so
$-q^{-n/2}[\#X(\F_{q^n})-(q^n+1)]$ must have the form $d\sqrt{p}$ where $d$ is a rational number.
Therefore
we must have 
\begin{align}\label{qn1}
-q^{-n/2}[\#X(\F_{q^n})-(q^n+1)]&=\sum_{d\mid u,\sqrt p \in  \mathbb Q(w_d)} \left( (e_{d,2}- e_{d,1})b_d\sqrt{p}\right).
\end{align}
By the same argument as in the previous paragraph when $n=m$,
we get that $-q^{-m/2}[\#X(\F_{q^m})-(q^m+1)]$ has the form $c\sqrt{p}$ where $c$ is a rational number.

Now we apply the same reasoning as above when $n=m$.
We get 
 \begin{align*}
-q^{-m/2}[\#X(\F_{q^m})-(q^m+1)]
&=\sum_{d\mid u,\sqrt p \in  \mathbb Q(w_d)} \left( e_{d,1}\sum_{j\in I_d^+} w_d^{-j} + e_{d,2}\sum_{j\in I_d^-} w_d^{-j}\right)\\
&+\sum_{d\mid u, \sqrt p \not \in  \mathbb Q(w_d)} \left( e_{d}\sum_{j\in I_d} w_d^{-j}  \right)\\
&=\sum_{d\mid u,\sqrt p \in  \mathbb Q(w_d)} \left( e_{d,1}(a_d+b_d\sqrt{p}) + e_{d,2}(a_d-b_d\sqrt{p})\right)\\
			&+\sum_{d\mid u, \sqrt p \not \in  \mathbb Q(w_d)} \left( e_{d}\cdot 2a_d \right)\\
			&=\sum_{d\mid u,\sqrt p \in  \mathbb Q(w_d)} \left( (e_{d,1}- e_{d,2})b_d\sqrt{p}\right).
\end{align*}  
Comparing this last expression with \eqref{qn1} we see that
\[
-q^{-n/2}[\#X(\F_{q^n})-(q^n+1)]
=-q^{-m/2}[\#X(\F_{q^m})-(q^m+1)]\cdot(-1).
\]

{\bf Case B.} When $p$  divides $t$.

Since $(t,u)=1$ and $p \mid t$, we have $p\nmid u$. Therefore,  $\sqrt{p}\not \in \mathbb Q(w_d)$ for each $d\mid u$ and so
\begin{equation}\label{M-odd-2}
M_X(z)=\prod_{d \mid u}\Phi_d(z)^{e_d}
\end{equation}
where $e_d$'s are non-negative integers. Hence, we have (by Equation \ref{eqn-root-sum-2} and \ref{M-odd-2}) 
\begin{align*} 
-q^{n/2}[\#X(F_{q^n})-(q^n+1)]&=\sum_{d\mid u} \left( e_{d}\sum_{j\in I_d} w_d^{-jt}  \right)\\&=\sum_{d\mid u} \left( e_{d}\sum_{j\in I_d} w_d^{-j}  \right)\\&=-q^{m/2}[\#X(F_{q^m})-(q^m+1)].
\end{align*}

\section{Proof of Theorem \ref{thm-l-poly}}

In this section we will give the proof of Theorem \ref{thm-l-poly}. The following 
well-known proposition will be useful to prove the theorem (see \cite{AFF-HS} for example). 

\begin{prop}\label{prop-L-Si}
Let $X$ be a supersingular curve of genus $g$ defined over $\mathbb F_q$. Let $L_X(T)= \sum_{i=0}^{2g} c_i T^i$ be the $L$-polynomial of $X$ and let $S_n=\#X(\mathbb F_{q^n})-(q^n+1)$ for all integers $n\ge 1$. Then $c_0=1$ and $$ic_i=\sum_{j=0}^{i-1}S_{i-j}c_j$$ for $i=1,\cdots,g$.
\end{prop}
\begin{proof}[Proof of Theorem \ref{thm-l-poly}]
	By Proposition \ref{prop-L-Si}, we have $S_1=c_1$ and $$S_i= ic_i-\sum_{j=1}^{i-1}S_{i-j}c_j$$  for $2\le i \le g$. Since $c_1,\cdots, c_{l-1}$ are known, we can inductively find $S_1, \cdots, S_{l-1}$ by above expression. Since $d=(l,s) < l$ and since we know what $S_d$ is, we can find $S_l$ by Theorem \ref{reduction-thm}. Therefore, since \begin{equation}\label{eqn-cl}
	c_l=\frac{1}{l}\sum_{j=0}^{l-1}S_{l-j}c_j 
	\end{equation}by Proposition \ref{prop-L-Si} and since $c_0,c_1,\cdots, c_{l-1}$ and $S_1,\cdots, S_l$ are known,  we can find $c_l$.   
\end{proof}

\section{Applications}

In this section we present a few applications of our results.

\subsection{An Example}

The usefulness of Theorems  \ref{reduction-thm} and \ref{thm-l-poly} is shown by an example.
Let is consider the curve $C: y^5-y=x^6$ over $\mathbb F_5$
and calculate its L-polynomial. 
This curve is actually the curve denoted $B_1^{(5)}$ in \cite{MY}.
and is a curve of genus $10$.
By Corollary 1 in \cite{MY} the roots of the L-polynomial are
 $\sqrt 5$ times a $4$-th root of unity (i.e., the period $s$ is 4). 
 Write $L_C$ as $$L_C(T)= \sum_{i=0}^{20} c_i T^i.$$ 
 Normally the values $c_i$ for $1 \le i\le 10$ have to be computed, and then the
 L-polynomial is determined.
 We will show using Theorem  \ref{thm-l-poly} that only three of the $c_i$ need to
 be computed in order to determine the entire L-polynomial. 
 
 The three values needed are $c_1$, $c_2$ and $c_4$, because 1, 2, and 4 are the
 divisors of the period.
 We calculate $c_1=0$, $c_2=-10$ and $c_4=-75$. 
 In order to find $L_C$ we have to find $c_3,c_5,c_6,c_7,c_8,c_9$ and $c_{10}$. 

Let us apply the recursion in Theorem \ref{thm-l-poly}. Firstly, we have $S_1=c_1=0$ and $S_2=2s_2-S_1c_1=-20$.

Next, we may apply Theorem  \ref{reduction-thm} to get $S_3$ since $(3,4)=1$.
 Thus  by 
Theorem \ref{reduction-thm} applied with the values
$s=4$, $n=3$, $m=1$, $r=1$, $t=3$, we get
\[
S_3=5^{(n-m)/2} \left( \frac{(-1)^{(t-1)/2}t}{5} \right) S_1 =0.
\]
Since we have $$3c_3=S_3c_0+S_2c_1+S_1c_2$$ by Proposition \ref{prop-L-Si}, we get $c_3=0$. 

Since we know $c_4=-75$ we get
$$S_4=4c_4-S_3c_1-S_2c_2-S_1c_1=-500.$$


In a similar way to $S_3$ we  find all $S_i$ for $i=5,\cdots, 10$ by Theorem \ref{reduction-thm}. 
These numbers are  $$S_5=S_7=S_9=0, \quad S_6=-500, \quad S_8=S_{10}=-12500.$$  
Then we can find all $c_i$ for $i=5,\cdots, 10$ by Theorem \ref{thm-l-poly}, inductively using the Equation \eqref{eqn-cl}.  We get $c_5=c_7=c_9=0$ 
and
$$c_6=\frac{1}{6}\sum_{j=0}^{5}S_{6-j}c_l=1000, \quad c_8=\frac{1}{8}\sum_{j=0}^{7}S_{8-j}c_j =1250, \quad  c_{10}=\frac{1}{10}\sum_{j=0}^{9}S_{10-j}c_j =-37500.$$ Therefore the $L$-polynomial of $C$ over $\mathbb F_5$ is 
\begin{align*}
1&-10T^2-75T^4+1000T^6+1250T^8-37500T^{10}\\ &+31250T^{12}+625000T^{14}-1171875T^{16}-3906250T^{18}+9765625T^{20}.
\end{align*}

\subsection{Families of Curves}

The authors have used the two main theorems of this paper in \cite{MY}
to calculate the exact number of rational points on curves
$y^p-y=x^{p^k+1}$ and $y^p-y=x^{p^k+1}+x$
over all extensions of $\F_p$.
The same techniques can be used for other supersingular curves.

Another example of applying these results can be found in 
our preprint \cite{MY2}. There we calculated the exact number of points
on $y^q-y=x^{q+1}-x^2$ in order to count the number of irreducible polynomials
with the first two coefficients fixed (providing another proof of a result of Kuzmin).

\subsection{Point Counting}

The results of this paper imply a speedup for point counting algorithms
for supersingular curves, at least in theory.
In general, to calculate the L-polynomial of a curve defined over $\F_q$,
one needs to compute the number of $\F_{q^i}$-rational points for all $i=1,2,\ldots ,g$.
However, Theorem 2 means that not all of these values are needed for supersingular curves.
The values that are needed are the number of $\F_{q^i}$-rational points where $i$ divides $s$.
As seen in the example above, we only needed four values 
($i=1,2,4$) instead of ten values.


\section{Value of the Period}

To apply the results of this paper one needs to know the period $s$ (or a multiple of $s$)
of a supersingular curve $X$ of genus $g$.
It is sometimes possible to find the period without finding the L-polynomial.
This is often the case for Artin-Schreier curves of the form
$y^p-y=xL(x)$ where $L(x)$ is a linearized polynomial.
Indeed, this is what the authors did in \cite{MY} for the curves
$y^p-y=x^{p^k+1}$ and $y^p-y=x^{p^k+1}+x$.
Another example of this is provided in \cite{KW}
with the hyperelliptic curves $y^2=x^{2g+1}+1$ where $(p,2g+1)=1$.
The period is shown to be the smallest $k$ such that $p^k \equiv -1 \pmod{4g+2}$.

For a simple supersingular abelian variety, we have $\phi (s)=2g$ or $4g$
(see \cite{MSZ}) where $\phi$ is the Euler phi-function.
Therefore if the Jacobian of $X$ is simple, we know that $\phi (s)=2g$ or $4g$.
This can be used to make a list of possible values of $s$.
If the Jacobian of $X$ is not simple it is isogenous to 
a product of simple abelian varieties of smaller dimension.

We call $X(\mathbb F_{q^n})$ \emph{minimal} if all the Weil numbers are $\sqrt{q^n}$.
Equivalently, $X(\mathbb F_{q^n})$ is minimal
if and only if $L_X(T)=(1-\sqrt{q^n} T)^{2g}$.
Thus, for a supersingular curve $X$ defined over $\F_q$, the period is the smallest extension degree over which $X$ is minimal.

\end{document}